\documentclass{amsart}
\newcommand{\ncmd}{\newcommand}
\ncmd{\dsum}{\displaystyle\sum}
\ncmd{\dprod}{\displaystyle\prod}
\ncmd{\ovl}{\overline}
\ncmd{\ptl}{\partial}
\ncmd{\rar}{\rightarrow}
\ncmd{\mA}{\script{A}}
\ncmd{\mC}{\mathbb{C}}
\ncmd{\mg}{\mathfrak{g}}
\ncmd{\mh}{\mathfrak{h}}
\ncmd{\mo}{\mathbbm{1}}
\newtheorem{Theorem}{Theorem}
\newtheorem{Lemma}{Lemma}
\begin{document}
\title{The Classical Family Algebra of the adjoint representation of $sl(n)$}
\author{Matthew Tai}
\maketitle
\section*{Table of Contents}
\begin{center}
Abstract\\
Notation\\
1: The Problem\\
2: Generators and Relations\\
3: Sufficiency of the Generators\\
4: Proof of the Relations\\
5: Natural Generators\\
6: Generalized Exponents\\
References
\end{center}
\section*{Abstract}
For the simplie Lie algebra $\mg = sl(n, \mC)$ we find a set of generators and relations for the classical family algebra $(End(\mg)\otimes S(\mg))^G$ as an algebra over the ring $I(\mg)$. From these we can then determine a $I(\mg)$-linear basis of the family algebra, and thus the generalized exponents of the irreducible components of $End(\mg)$ viewed as a $\mg$-module.
\section*{Notation}
\noindent $\mg$ a simple Lie algebra over $\mC$\\
$G$ the corresponding adjoint Lie group to $\mg$\\
Unless marked with a weight, $\pi$ indicates the defining representation.\\
$\omega_i$ are fundamental weights of $\mg$, $\omega_1$ corresponds to the defining representation $\pi$
All traces are taken with respect to the defining representation.\\
$x_\alpha,x_\beta,\ldots$ Greek letters indicate elements of $\mg$\\
$k,l,n,d$ Roman letters indicate integers\\
$\{x_\alpha\}$ a basis of $\mg$, $\{x^\beta\}$ the dual basis such that $\left\langle x^\beta,x_\alpha\right\rangle = \delta_\alpha^\beta$\\
$K_{\alpha\beta} = tr(\pi(x_\alpha)\pi(x_\beta))$. $K^{\alpha\beta}$ is the corresponding entry in the matrix inverse of $K$. $K_{\alpha\beta}$ is the Killing form up to a scaling, so we simply refer to both as the Killing form.\\
$\ptl^\alpha = \frac{\ptl}{\ptl x_\alpha}$, considering $x_\alpha$ as coordinate functions on $\mg^*$
\section*{1: The Problem}
Given a simple Lie algebra $\mg$ over $\mC$, we can consider the universal enveloping algebra $U(\mg)$ as a $\mg$-module under the action induced from the adjoint action on itself. As a $\mg$-module, $U(\mg)$ is isomorphic to $Sym(\mg) = Pol(\mg^*)$, the polynomial functions on $\mg^*$, and in turn $\mC[\mg^*]$, $\mC$ adjoin the basis elements of $\mg$, treated as commuting variables.\\
As a $\mg$-module, $\mC[\mg^*]$ splits as $I(\mg)\otimes H(\mg)$, where $I(\mg) = \mC[\mg^*]^G$ is the subalgebra of polynomials invariant under the adjoint action, and $H(\mg)$ is the subalgebra of polynomials annihilated by all left-invariant differential operators on $G = \exp(\mg)$ with constant coefficients and no constant term. The structure of $I(\mg)$ is well-known [11], so the interesting question is about $H(\mg)$.\\
As $H(\mg)$ is a graded $\mg$-module, it decomposes into homogeneous irreducible finite-dimensional $\mg$-modules, and the degrees in which a module of a given isotopy type appears are called the generalized exponents of that module. Thus finding the generalized exponents of the modules gives insight into the structure of $H(\mg)$, and thus into $\mC[\mg^*], Sym(\mg)$ and $U(\mg)$ [7].\\
The generalized exponents can be calculated via the $q$-analogue of Kostant's multiplicity formula ([3,4,5,8,10]), expanding modules via resolutions by Verma modules, but this is unwieldy due to needing to sum over the Weyl group of $\mg$ and the calculation of the Kostant partition function [6]. Instead, Kirillov found a set of algebras that contain finitely many different isotypic components as $\mg$-modules, so that the structures of these algebras could yield another way to calculate the generalized exponents.\\
Given a finite dimensional representation $(\pi_\lambda, V_\lambda)$ of $\mg$ we can look at the endomorphisms of $V_\lambda$ as a vector space, $End(V_\lambda)$. $\pi_\lambda$ induces a representation $(\Pi_\lambda,V_\lambda)$ of $G$. Then $G$ acts on $End(V_\lambda)$ by an action $\widetilde{\pi_\lambda}(g)(A) = \Pi_\lambda(g)A\Pi_\lambda(g^{-1})$ and on $\mC[\mg^*]$ by an action induced from the adjoint action on $\mg$, and hence acts on $End(V_\lambda)\otimes \mC[\mg^*]$. In [6], Kirillov defines the classical family algebra 
$$C_\lambda(\mg) = (End(V_\lambda)\otimes \mC[\mg^*])^G$$ 
the elements of $End(V_\lambda)\otimes \mC[\mg^*]$ invariant under this action. The invariant polynomials $I(\mg)$ inject into this algebra as $Id_{V_\lambda}\otimes P$ for $P \in I(\mg)$, and hence we get that $C_\lambda(\mg)$ is a free $I(\mg)$-module. The image of $I(\mg)$ is central so we treat $I(\mg)$ as the scalar ring of the family algebra.\\
We can consider $End(V_\lambda)$ as a $\mg$-module, and hence it decomposes into a number of irreducible $\mg$-modules $End(V_\lambda)_\mu$ with highest weight $\mu$. An element $A \in C_\lambda(\mg)$ is called harmonic if there is a $\mu$ such that $A \in (End(V_\lambda)_\mu\otimes H(\mg))^G$ and $A$ is homogeneous. Since $A$ is $\mg$-invariant, its polynomial part is in $H(\mg)_{w_{max}(-\mu)}$, which is isomorphic as an $\mg$-module to $H(\mg)_\mu$ via the Killing form. Hence the existence of a non-zero harmonic element $A$ with corresponding weight $\mu$ implies the existence of a $V_\mu$-isotypic component of $H(\mg)$ that is homogeneous with degree $\deg(A)$. In fact, a $I(\mg)$-linear basis of $C_\lambda(\mg)$ comprised of harmonic elements has exactly one basis element for each generalized exponent of the subrepresentations of $End(V_\lambda)$.\\
In this paper we will be finding the algebraic structure of $C_{\omega_1+\omega_{n-1}}(sl(n,\mC))$, the classical family algebra for the adjoint representation ($V = \mg$) of the algebras $sl(n) = A_{n-1}$. For $n = 2$ and $3$, the structure of the family algebra has already been worked out by [9], so this paper will assume $n\geq 4$. As an example, we will use $n=4$.
\section*{2: Generators and Relations}
Consider first the family algebra for $\pi_\lambda = \pi$, the defining representation; thus $V$ is $n$-dimensional, $End(V)$ is the set of $n\times n$ matrices and the family algebra consists of polynomial-valued $n\times n$ matrices. In this algebra define $F = \pi(x_\alpha)\otimes K^{\alpha\beta}x_\beta$, an $n\times n$ matrix whose entries are degree $1$ polynomials in $\mg^*$. Note that $tr(F^2) = K^{\alpha\beta}x_\alpha x_\beta$. We then define the symmetrized Casimir elements $c_k = tr(F^k)$ and, using the Cayley-Hamilton theorem, define $d_k$ by
$$F^n = \dsum_{k=0}^{n-2}d_{n-k}F^k$$
Note that $c_0 = n$, $c_1 = d_{n-1} = 0$, and that we can go from the $c_k$ to the $d_{k}$ via the Newton formulae.  By [11], $I(\mg) = \mC[c_2,\ldots,c_n]$. For the rest of this paper, $c_k$ is considered to have degree $k$.\\
${ }$\\
Example: for $n = 4$, $d_1 = 0, d_2 = \frac{c_2}{2}$, $d_3 = \frac{c_3}{3}$, $d_4 = \frac{c_4}{4}-\frac{c_2^2}{8}$.\\
${ }$\\
Switching now to the family algebra for the adjoint representation, we have the isomorphism $End(\mg) \cong \mg\otimes \mg^*$ and hence elements of the family algebra can be written as 
$$x_\alpha\otimes x^\beta \otimes P_\beta^\alpha$$ 
for $x_\alpha \in \mg$, $x^\beta \in \mg^*$, and hence $x_\alpha\otimes x^\beta \in End(\mg)$, and $P_\beta^\alpha \in \mC[\mg^*]$. Using this isomorphism, we can refer to the $x_\alpha$ and $x^\beta$ as coordinates for the entries in a matrix in $End(\mg)$. We can also write an element of the family algebra just by specifying the polynomial part, $\{P_\beta^\alpha\}$, with the coordinates implied by the uncontracted indices on the polynomial part. Multiplication of elements of the algebra is given by
$$(x_\alpha\otimes x^\beta \otimes P_\beta^\alpha)(x_\gamma\otimes x^\delta\otimes Q^\gamma_\delta) = x_\alpha\otimes x^\delta \otimes P_\beta^\alpha Q_\delta^\beta$$
or, in terms of just the polynomial data,
$$(\{P_\beta^\alpha\}\{Q_\delta^\gamma\})_\zeta^\epsilon = \{P^\epsilon_\eta Q^\eta_\zeta\}$$
If every polynomial in the polynomial part of an element of the family algebra is of degree $k$, we say that the elemet of the family algebra has degree $k$.
\begin{Theorem}\textbf{(Generators)} The following three elements generate the family algebra over $I(\mg)$:
$$L = x_\alpha\otimes x^\beta \otimes K^{\alpha\gamma}tr(\pi(x_\beta) \pi(x_\gamma) F)$$
$$R = x_\alpha\otimes x^\beta \otimes K^{\alpha\gamma}tr(\pi(x_\gamma) \pi(x_\beta) F)$$
$$S = x_\alpha\otimes x^\beta \otimes K^{\alpha\gamma}x_\beta x_\gamma $$
where the traces are taken considering $\pi(x_\beta), \pi(x_\gamma)$ and $F$ as polynomial-valued $n\times n$ matrices.
\end{Theorem}
The polynomial parts of $L$ and $R$ are degree $1$ polynomials, as $\pi(x_\beta)$ and $\pi(x_\gamma)$ both have scalar entries and $F$ takes entries in degree $1$ polynomials. The polynomial parts of $S$ are degree $2$ polynomials.\\
For the relations, first we define for concision the following notation:\\
Let
$$L_k = L^k + \frac{1}{n}\left(\dsum_{\lambda\vdash k,|\lambda| > 1}\left(\dprod_{i>2}^{|\lambda|}\frac{c_{\lambda_i}}{n}\right)L^{\lambda_1-1}SR^{\lambda_2-1}\right)$$
$$R_k = R^k + \frac{1}{n}\left(\dsum_{\lambda\vdash k,|\lambda| > 1}\left(\dprod_{i>2}^{|\lambda|}\frac{c_{\lambda_i}}{n}\right)L^{\lambda_1-1}SR^{\lambda_2-1}\right)$$
where $\lambda$ ranges over compositions. As will be shown below, 
$$L_k = x_\alpha\otimes x^\beta \otimes K^{\alpha\gamma}tr(\pi(x_\beta)\pi(x_\gamma)F^k)$$
$$R_k = x_\alpha\otimes x^\beta \otimes K^{\alpha\gamma}tr(\pi(x_\gamma)\pi(x_\beta)F^k)$$
${ }$\\
Example: for $n = 4$, 
$$L_1 = L,\text{ }  R_1 = R, \text{ } L_2 = L^2+\frac{1}{4}S, \text{ } R_2 = R^2+\frac{1}{4}S$$
$$L_3 = L^3+\frac{1}{4}(LS+SR), \text{ } R_3 = R^3+\frac{1}{4}(LS+SR)$$
$$L_4 = L^4+\frac{1}{4}(L^2S+LSR+SR^2)+\frac{1}{16}c_2S$$
$$R_4 = R^4+\frac{1}{4}(L^2S+LSR+SR^2)+\frac{1}{16}c_2S$$
${ }$\\
For all $n$, we get the following relations:
\begin{Theorem}\textbf{(Relations)} The following relations are sufficient for defining the family algebra with the generators above
$$LR = RL, \ LS = RS, \ SL = SR$$
$$SL^kR^lS = -n\dsum_{\lambda\vdash k+l+2} \dprod_i^{|\lambda|} \left(-\frac{c_{\lambda_i}}{n}\right) S$$
$$L_n = \dsum_{k=0}^{n-2} d_{n-k}L_k, \ R_n = \dsum_{k=0}^{n-2} d_{n-k}R_k$$
$$\dsum_{l=0}^{n-1}L^{n-1-l}R^l=\dsum_{k=0}^{n-3}d_{n-k-1}\dsum_{l=0}^kL^{k-l}R^l$$
\end{Theorem}
Note that these relations are not independent. The $SL^kR^lS$ relations become redundant when $k+l$ exceeds $n-1$. The $L_n$ relation minus the $R_n$ relation gives the $L^kR^l$ relation times $L-R$, and hence the three relations are not independent. Here I give both the $L_n$ and $R_n$ relation because each of them is easier to prove individually than any linear combination of them that isn't a multiple of the $L^kR^l$ relation.\\
${ }$\\
Example: for $n = 4$, we get
$$SS = c_2S, \text{ } SLS = SRS = c_3S, \text{ } SL^2S = SLRS = SR^2S = (c_4-c_2^2/4)S$$
$$SL^3S = SL^2RS = SLR^2S = SR^3S = (c_5-\frac{2}{4}c_2c_3)S = \frac{c_2c_3}{3}S$$
$$L^4+\frac{1}{4}(L^2S+LSR+SR^2)+\frac{1}{16}c_2S = \frac{c_2}{2}(L^2+\frac{1}{4}S)+\frac{c_3}{3}L+\frac{c_4}{4}-\frac{c_2^2}{8}$$
$$R^4+\frac{1}{4}(L^2S+LSR+SR^2)+\frac{1}{16}c_2S = \frac{c_2}{2}(R^2+\frac{1}{4}S)+\frac{c_3}{3}R+\frac{c_4}{4}-\frac{c_2^2}{8}$$
$$L^3+L^2R+LR^2+R^3 = \frac{c_2}{2}(L+R)+\frac{c_3}{3}$$
\section*{3: Sufficiency of the Generators}
Note that for $T$ a $(1,k+1)$ tensor built from the Killing form, structure constants and Casimir elements and $P \in \mC[\mg^*]^G$, the expression 
$$x_\alpha\otimes x^\beta \otimes T_{\beta\gamma_1\ldots \gamma_k}^\alpha \ptl^{\gamma_1}\ldots\ptl^{\gamma_k} P$$ will always be invariant and hence an element of the family algebra, since $\ptl^\alpha$ transforms as a $(1,0)$ tensor.\\
We define $(x_\alpha)^b_a$ to be the $a,b$ entry of $\pi(x_\alpha)$. We define the object 
$$t_{\alpha_1\alpha_2\cdots \alpha_k} = tr(\pi(x_{\alpha_1})\pi(x_{\alpha_2})\cdots\pi(x_{\alpha_k})) = (x_{\alpha_1})_{a_1}^{a_2}(x_{\alpha_2})_{a_2}^{a_3}\cdots (x_{\alpha_k})_{a_k}^{a_1}$$
so that
$$c_k = tr(F^k) = t_{\alpha_1\alpha_2\cdots \alpha_k}\prod_i K^{\alpha_i\beta_i}x_{\beta_i}$$
Note that the structure constants, defined by $[x_\alpha,x_\beta] = f_{\alpha\beta}^\gamma x_\gamma$, can be written as 
$$f_{\alpha\beta}^\gamma = CK^{\delta\gamma}(t_{\alpha\beta\delta}-t_{\beta\alpha\delta})$$ 
for some constant $C$. Also note that any trace that isn't fully symmetrized can be expressed in terms of symmetrized traces and structure constants [1] so all of the invariants can be written in terms of traces and the Killing form. Hence our generators, built from traces and the Killing form, are all in the family algebra.\\
For any Lie algebra $\mg$, we have a projection operator from $End(V) = V\otimes V^*$ to the adjoint subspace, given by
$$Proj_{a,c}^{b,d}(v_b\otimes v^a) = v_d\otimes v^c \in \mg$$
For $sl(n)$, this projector 
$$Proj_{a,c}^{b,d} = (x_\alpha)_a^bK^{\alpha\beta}(x_\beta)_c^d = \delta_a^d\delta_c^b - \frac{1}{n}\delta_a^b\delta_c^d$$ 
which can be read as removing the trace/$n$ times the identity  from operators from $V$ to $V$, i.e. the standard projection from $gl(n) = End(V)$ to $sl(n)$.\\
${ }$\\
The proof that the generators given generate the whole algebra follows from two lemmas.
\begin{Lemma}
Any expression of traces contracted to each other via the Killing form can be expressed as a polynomial in uncontracted traces
\end{Lemma}
\begin{proof}
If we have $t_{\alpha_1\cdots\alpha_k}K^{\alpha_1\beta_1}t_{\beta_1\cdots\beta_l}$, we can use the projection operator to get 
\begin{align*}
t_{\alpha_1\cdots\alpha_k}K^{\alpha_1\beta_1}t_{\beta_1\cdots\beta_l} = & \ (x_{\alpha_1})_{a_1}^{a_2}\cdots (x_{\alpha_k})_{a_k}^{a_1}K^{\alpha_1\beta_1}(x_{\beta_1})_{b_1}^{b_2}\cdots (x_{\beta_l})_{b_l}^{b_1}\\
= & \ (\alpha_2)_{a_2}^{a_3}\cdots (x_{\alpha_k})_{a_k}^{a_1}(x_{\beta_2})_{a_1}^{b_3}\cdots (x_{\beta_l})_{b_l}^{a_2}\\ 
&\ - \frac{1}{n}(\alpha_2)_{a_1}^{a_3}\cdots (x_{\alpha_k})_{a_k}^{a_1}(x_{\beta_2})_{b_1}^{b_3}\cdots (x_{\beta_l})_{b_l}^{a_1}\\
= & \ t_{\alpha_2\cdots \alpha_k\beta_2\cdots\beta_l}-\frac{1}{n}t_{\alpha_2\cdots \alpha_k}t_{\beta_2\cdots\beta_l}
\end{align*}
A contraction between a trace and itself can also be rewritten, giving
$$t_{\alpha_1\cdots \alpha_i\cdots \alpha_k}K^{\alpha_1\alpha_i} = t_{\alpha_2\cdots\alpha_{i-1}}t_{\alpha_{i+1}\cdots\alpha_k}-\frac{1}{n}t_{\alpha_2\cdots\alpha_{i-1}\alpha_{i+1}\cdots\alpha_k}$$
Hence any expression in traces contracted with the Killing form which has only lower indices can be rewritten as a polynomial in traces.
\end{proof}
So we assume from now on that no traces are contracted to any other traces.\\
To get from a tensor built from traces to an element of the family algebra, we simply pick one index to be $x_\alpha$, pick a second index $x^\beta$, raising it with the Killing form, and then contract the rest of the indices to elements of $\mg$ and symmetrize them.
\begin{Lemma}
$$L_k = \{K^{\alpha\gamma}tr(\pi(x_\beta)\pi(x_\gamma)F^k)\}$$
$$R_k = \{K^{\alpha\gamma}tr(\pi(x_\gamma)\pi(x_\beta)F^k)\}$$
\end{Lemma}
\begin{proof}
For $k = 1$, the desired identity follows by definition.\\
For $k = 2$:
$$L^2 = \{K^{\alpha\gamma}tr(\pi(x_\beta)\otimes \pi(x_\delta)F)K^{\delta\epsilon}tr(\pi(x_\epsilon)\pi(x_\gamma)F\}$$
Replacing the $x_\delta K^{\delta\epsilon}x_\epsilon$ with the adjoint representation projector, we get
$$\{K^{\alpha\gamma}tr(\pi(x_\beta)\pi(x_\gamma)F^2)\}-\{\frac{1}{n}K^{\alpha\gamma}tr(\pi(x_\beta)F)tr(\pi(x_\gamma)F)\}$$
The second term simplifies to $\{\frac{1}{n}K^{\alpha\gamma}x_\beta x_\gamma\} = \frac{1}{n}S$. Hence we get that 
$$L_2 = \{K^{\alpha\gamma}tr(\pi(x_\beta)\pi(x_\gamma)F^2)\}, \text{ } R_2 = \{K^{\alpha\gamma}tr(\pi(x_\beta)\pi(x_\gamma)F^2)\}$$
Now suppose that for all $l < k$,
$$L_{l} = \{K^{\alpha\gamma}tr(\pi(x_\beta)\pi(x_\gamma)F^{l})\}, \text{ } R_{l} = \{K^{\alpha\gamma}tr(\pi(\pi(x_\gamma)x_\beta)F^l)\}$$
We write $L^k = LL^{k-1}$, which we can expand as
$$LL_{k-1} - L\left(\dsum_{\lambda\vdash k-1,|\lambda| > 1}\frac{1}{n}\left(\dprod_{i>2}^{|\lambda|}\frac{c_{\lambda_i}}{n}\right)L^{\lambda_1-1}SR^{\lambda_2-1}\right)$$
Considering the first term, we get
$$LL_{k-1} = \{K^{\alpha\gamma}tr(\pi(x_\delta)\pi(x_\gamma) F)K^{\delta\epsilon}tr(\pi(x_\beta)\pi(x_\epsilon)F^{k-1})\}$$
Expanding out $x_\delta K^{\delta\epsilon}x_\epsilon$ by the adjoint representation projector gives
$$\{K^{\alpha\gamma}tr(\pi(x_\beta)\pi(x_\gamma)F^k)\}-\{\frac{1}{n}K^{\alpha\gamma} tr(\pi(x_\beta)F)tr(\pi(x_\gamma)F^{k-1})\}$$
Since the identity holds for $k-2$, we get
\begin{align*}SR_{k-2} &= \{K^{\alpha\gamma}x_\gamma x_\beta\}\{K^{\delta\epsilon}tr(\pi(x_\epsilon)\pi(x_\zeta)F^{k-2})\}\\
&= \{K^{\alpha\gamma}x_\gamma x_\beta K^{\beta\epsilon}tr(\pi(x_\epsilon)\pi(x_\zeta)F^{k-2})\}\\
&= \{K^{\alpha\gamma}x_\gamma tr(F\pi(x_\zeta)F^{k-2})\}\\
&= \{K^{\alpha\gamma}x_\gamma tr(\pi(x_\zeta)F^{k-1})\}
\end{align*}
Hence the second term is $SR_{k-2}$.\\
Our expression for $\{K^{\alpha\gamma}tr(\pi(x_\beta)\pi(x_\gamma)F^k)\}$ is now $L^k$ plus $$SR_{k-2}+L\dsum_{\lambda\vdash k-1,|\lambda| > 1}\frac{1}{n}\left(\dprod_{i>2}^{|\lambda|}\frac{c_{\lambda_i}}{n}\right)L^{\lambda_1-1}SR^{\lambda_2-1}$$ 
These two expressions become 
$$\dsum_{\lambda\vdash k,|\lambda| > 1}\frac{1}{n}\left(\dprod_{i>2}^{|\lambda|}\frac{c_{\lambda_i}}{n}\right)L^{\lambda_1-1}SR^{\lambda_2-1}$$
since the second contains exactly the terms that contain an $L$ and $SR_{k-2}$ contains exactly the terms that do not.\\
Hence we get that 
$$\{K^{\alpha\gamma}tr(\pi(x_\beta)\pi(x_\gamma)F^k)\}= L^k+\dsum_{\lambda\vdash k-1,|\lambda| > 1}\frac{1}{n}\left(\dprod_{i>2}^{|\lambda|}\frac{c_{\lambda_i}}{n}\right)L^{\lambda_1-1}SR^{\lambda_2-1} = L_k$$ as desired, and similarly for $R_k$.
\end{proof}
\begin{proof} (Theorem 1)\\
By Lemma 1, we only need to worry about where $x_\alpha$ and $x^\beta$ are contracted to; everything else is traces. For a single term, the coordinate indices $x_\alpha$ and $x^\beta$ are either contracted to different traces, or the same trace. Assume there are no nontrivial Casimir elements involved.\\ 
In the first case, our object is of the form $\{K^{\alpha\gamma}tr(\pi(x_\gamma) F^k)tr(\pi(x_\beta) F^l)\}$. By Lemma 2 we can write this object as $L_{k-1}SR_{l-1}$. Note that $k$ and $l$ are both at least $1$, since if $k = 0$ then we'd get $tr(\pi(x_\gamma)F^k) = tr(\pi(x_\gamma)) = 0$, and similarly if $l = 0$.\\
Now suppose that the coordinates $x_\alpha$ and $x^\beta$ are contracted to the same trace. Since all the other indices are contracted to elements of $\mg$, by Lemma 2 we get the object $\{K^{\alpha\gamma}tr(\pi(x_\gamma) F^k \pi(x_\beta) F^l)\}$, which we can then write as $L_kR_l+\frac{1}{n}L_{k-1}SR_{l-1}$ via the adjoint representation projector.\\
Hence $L,S$ and $R$ generate the entire family algebra.
\end{proof}
\section*{4: Proof of the Relations}
First we prove that the relations hold:
\begin{proof} (Verifying the Relations)
\begin{align*}
LR &= \{K^{\alpha\gamma}tr(x_\gamma x_\beta F)\}\{K^{\delta\zeta}tr(x_\epsilon x_\zeta F)\}\\
&= \{K^{\alpha\gamma}tr(x_\gamma x_\beta F)K^{\beta\zeta}tr(x_\epsilon x_\zeta F)\}\\
&= \{K^{\alpha\gamma}tr(x_\gamma F x_\epsilon F)-\frac{1}{n}K^{\alpha\gamma}tr(x_\gamma F)tr(x_\epsilon F)\}\\
&= \{K^{\alpha\gamma}tr(F x_\gamma F x_\epsilon)-\frac{1}{n}K^{\alpha\gamma}tr(F x_\gamma)tr(F x_\epsilon)\}\\
&= \{K^{\alpha\gamma}tr(F x_\beta x_\gamma)K^{\beta\zeta}tr(F x_\zeta x_\epsilon\})\\
&= \{K^{\alpha\gamma}tr(x_\beta x_\gamma F)K^{\beta\zeta}tr(x_\zeta x_\epsilon F\})\\
&= \{K^{\alpha\gamma}tr(x_\beta x_\gamma F)\}\{K^{\delta\zeta}tr(x_\zeta x_\epsilon F)\}\\
&= RL
\end{align*}
\begin{align*}
LS &= \{K^{\alpha\gamma}tr(x_\gamma x_\beta F)\}\{K^{\delta\zeta}x_\epsilon x_\zeta\}\\
&= \{K^{\alpha\gamma}tr(x_\gamma x_\beta F)K^{\beta\zeta}x_\epsilon x_\zeta\}\\
&= \{K^{\alpha\gamma}tr(x_\gamma F^2)x_\epsilon\}\\
&= \{K^{\alpha\gamma}tr(F x_\gamma F)x_\epsilon\}\\
&= \{K^{\alpha\gamma}tr(x_\beta x_\gamma F)K^{\beta\zeta}x_\epsilon x_\zeta\}\\
&= \{K^{\alpha\gamma}tr(x_\beta x_\gamma F)\}\{K^{\delta\zeta}x_\epsilon x_\zeta\}\\
&= RS
\end{align*}
The relation $SL = SR$ follows analogously.\\
Since $RS = LS$ and $LR = RL$, $SL^kR^lS = SL^{k+l}S$. So we only prove the relation for this expression:
\begin{align*}
SL^kS &= \{K^{\alpha\gamma}x_\beta x_\gamma\}\{K^{\delta_1\zeta_1}tr\big(\pi(x_{\zeta_1})\pi(x_{\epsilon_1})F\big)\}\cdots\{K^{\delta_k\zeta_k}tr\big(\pi(x_{\zeta_k})\pi(x_{\epsilon_k}) F\big)\}\{K^{\eta\iota}x_\theta x_\iota\}\\
&= \{K^{\alpha\gamma}x_\beta x_\gamma K^{\beta\zeta_1}tr\big(\pi(x_{\zeta_1})\pi(x_{\epsilon_1})F\big)K^{\epsilon_1\zeta_2}\cdots K^{\epsilon_{k-1}\zeta_k}tr\big(\pi(x_{\zeta_k})\pi(x_{\epsilon_k})F\big)K^{\epsilon_k\iota}x_\theta x_\iota\}\\
&= \{K^{\alpha\gamma}x_\gamma x_\theta (tr\big(\pi(x_{\epsilon_1}) F^2\big)K^{\epsilon_1\zeta_2}tr\big(\pi(x_{\zeta_1})\pi(x_{\epsilon_2})F\big)\cdots K^{\epsilon_{k-1}\zeta_k}tr\big(\pi(x_{\zeta_k})F^2\big))\}\\
&= (tr\big(\pi(x_{\epsilon_1}) F^2\big)K^{\epsilon_1\zeta_2}tr\big(\pi(x_{\zeta_1})\pi(x_{\epsilon_2})F\big)K^{\epsilon_2\zeta_3}\cdots K^{\epsilon_{k-1}\zeta_k}tr\big(\pi(x_{\zeta_k})F^2\big))S
\end{align*}
where in the third equality we replace $x_\beta K^{\beta\zeta_1}\pi(x_{\zeta_1})$ and $\pi(x_{\epsilon_k}) K^{\epsilon_k\iota}x_\iota$ with $F$.\\ 
Now we just need to expand out the scalar:
$$tr\big(\pi(x_{\epsilon_1}) F^2\big)K^{\epsilon_1\zeta_2}tr\big(\pi(x_{\zeta_1})\pi(x_{\epsilon_2})F\big)K^{\epsilon_2\zeta_3}\cdots K^{\epsilon_{k-1}\zeta_k}tr\big(\pi(x_{\zeta_k}) F^2\big)$$
Each instance of the Killing form in the previous expression gives rise to an expansion via the adjoint representation projector, either connecting a pair of $F$s into a single trace or separating them into a product of two traces. Hence for each composition of $k+2$ we get a product of Casimir elements corresponding to that composition, where separations give a factor of $-\frac{1}{n}$. Thus we get the desired relation.\\
The next relation follows from Lemma 2. Since $d_{n-k}$ is a scalar, we can move it in and out of traces and hence by the Cayley-Hamilton theorem applied to $F$,
\begin{align*}
L_n&= \{K^{\alpha\gamma}tr\big(\pi(x_\beta)\pi(x_\gamma)F^n\big)\}\\
&= \{K^{\alpha\gamma}tr\big(\pi(x_\beta)\pi(x_\gamma)\dsum_k d_{n-k}F^k\big)\}\\
&= \dsum_k d_{n-k}\{K^{\alpha\gamma}tr\big(\pi(x_\beta)\pi(x_\gamma) F^k\big)\}\\
&= \dsum_k d_{n-k}L_k
\end{align*}
The same argument gives the $R_n$ relation.\\
The last relation follows from the fact that for $sl(n)$ the highest-degree algebraically independent Casimir element has degree $n$ so $c_{n+1}$ must reduce to a polynomial in $c_k$ for $k \leq n$. Hence the object written as $\{K_{\beta\gamma} \ptl^\alpha \ptl^\gamma c_{n+1}\}$ must be reducible. Writing out the full expression gives the $L^kR^l$ reduction relation.
\end{proof}
\begin{proof} (Sufficiency of the Relations)\\
Because of the relations $LR-RL = 0$, $(L-R)S = S(L-R) = 0$, the $SL^kR^lS$ relation and the $L_n$ and $R_n$ relations, all of our $I(\mg)$ linearly independent monomials can be written in the form $L^kS^mR^l$ where $k < n$, since otherwise we could use the $L_n$ relation to reduce it, $l < n$ for the same reason, and $m$ either $0$ or $1$ since the $SL^kR^lS$ relation reduces any pair of $S$ factors to a single $S$.
\begin{Lemma}
No monic polynomial in $L+R$ with coefficients in $I(\mg)$ and degree less than $n-1$ can vanish. 
\end{Lemma}
\begin{proof}
If we consider $L^kR^{m-k}$, lower the raised coordinate using the Killing form, and then symmetrize over all of the indices, coordinate or otherwise, we end up with a polynomial in Casimir elements of degree $m+2$ including a term of $c_{m+2}$ and all other terms products of Casimir elements of lower degree.\\
Now suppose that we have a monic polynomial in $L+R$ with coefficients in $I(\mg)$ and degree $m$ less than $n-1$. Then the leading terms, i.e. the terms involving no nontrivial Casimir elements, has positive coefficients for all terms of the form $L^kR^{m-k}$ and thus the symmetrization of this polynomial then yields a polynomial in Casimir elements with nonvanishing $c_{m+2}$ coefficient.\\
Since $m < n-1$, we have that $m+2 < n+1$ and hence $c_{m+2}$ is algebraically independent of the Casimir elements of lower degree; hence the symmetrization cannot vanish, and hence the polynomial in $(L+R)$ cannot vanish.
\end{proof}
Consider now the $L^kR^l$ relation. The leading term is 
$$\dsum_k L^kR^{n-1-k}$$
and hence multiplying this leading term by $(L+R)^m$ yields a polynomial in $L$ and $R$ that only has positive cofficients. In particular, the term $L^{n-1}R^m$ has positive coefficient in this polynomial. For $m \leq n-1$, neither $L^{n-1}$ nor $R^m$ can be reduced by the $L_n$ or $R_n$ relations.\\ 
Hence we get that for $m < n-1$, $(L+R)^m$ times the $L^kR^l$ relation yields a relation in each degree greater than $n-2$ that cannot be deduced from the other relations. Since no polynomial in $L+R$ vanishes for degree less than $n-1$, we get that the relations of the form $(L+R)^m$ times the $L^kR^l$ relation are themselves linearly independent from one another over $I(\mg)$. We also get a relation in degree $2n-2$ by squaring the $L^kR^l$ relation, and this one is also linearly independent from the other relations since the $L^kR^l$ relation is itself a polynomial in $L$ and $R$ that is linearly independent from all polynomials in $L+R$, just by comparing leading coefficients.\\
Now we show that there cannot be any other relations that are not in the ideal generated by the ones listed. We do so by counting the number of $I(\mg)$-linearly independent monomials.\\
Note that since $LS = RS$, 
$$L^{n-1}S = \frac{1}{n}\dsum_k L^kR^{n-1-k}S$$
Hence, using the $L^kR^l$ relation, we can reduce $L^{n-1}S$ to terms involving nontrivial Casimir elements. Hence we get that $L^{n-1}SR^l$ is not linearly independent over $I(\mg)$ from terms of lower degree. Similarly, $L^kSR^{n-1}$ cannot be linearly independent.\\
Thus we get that our linearly independent monomials are $L^kSR^l$ for $0\leq k,l\leq n-2$ and $L^kR^l$ for $0 \leq k, l\leq n-1$, minus one in each degree between $n-1$ and $2n-2$ since $(L+R)^m$ times the $L^kR^l$ relation gives $L^{n-1}R^m$ in terms of other monomials.\\ 
This yields a total of $2n^2-3n+1$ terms not known to be linearly dependent. If there are more relations, then there will be fewer linearly independent terms.\\
In [6], we get that the dimension over $I(\mg)$ of the family algebra $V_\lambda$ is given by
$$\dsum_{\mu} m_\lambda(\mu)^2$$
where $m_\lambda(\mu)$ is the multiplicity of the weight $\mu$ in $V_\lambda$. Since $I(\mg)$ is an integral domain and the family algebra is free over it, there is a set of linearly independent objects of size equal to the dimension.\\
For the adjoint representation, the weights with non-zero multipllicity are the roots, each with multiplicity $1$, and $0$, with multiplicity equal to the rank of the algebra. This gives us $n(n-1)+(n-1)^2 = 2n^2-3n+1$. Hence, since the relations given above limit us to $2n^2-3n+1$ linearly independent elements at most, and any further relations would reduce that number, there cannot be any more relations.
\end{proof}
Example: for $n = 4$, we get that the relations give us the following $21$ linearly independent pieces:
$$1, L, R, L^2, LR, R^2, S, L^3, L^2R, LR^2, LS, SR$$ 
$$L^3R, LR^3, L^2S, LSR, SR^2, L^3R^2, L^2SR, LSR^2, L^2SR^2$$
as predicted by the dimension formula $2\cdot4^2-3\cdot 4+1 = 21$
\section*{5: Natural Generators}
Instead of using $L$ and $R$ as defined above, we can instead look at $M = \frac{1}{2}(L-R)$ and $N = \frac{1}{2}(L+R)$. These elements are in many senses more natural, as $M$ is proportional to $x_\alpha\otimes x^\beta \otimes K^{\alpha\gamma}f_{\gamma\beta}^\delta x_\delta$, with $f_{\gamma\beta}^\delta$ being the structure constant, and $N$ is proportional to $x_\alpha\otimes x^\beta \otimes K_{\beta\gamma}\ptl^\alpha\ptl^\gamma c_3$. In particular, $M$ and $N$ are harmonic.\\
In this basis, we get the following relations:
$$MN = NM, \text{ } MS = SM = 0$$
$$SN^kS = -n\dsum_{\lambda\vdash k+2} \dprod_i^{|\lambda|} \left(-\frac{c_{\lambda_i}}{n}\right)S$$
$$\dsum_{k=0}^{\left\lfloor\frac{n-1}{2}\right\rfloor}\binom{n}{2k+1}N^{n-1-2k}M^{2k}=\dsum_{j=1}^{n-2}d_{n-j}\dsum_{k=0}^{\left\lfloor\frac{j-1}{2}\right\rfloor}\binom{j}{2k+1}N^{j-1-2k}M^{2k}$$
Similarly to $L_k$ and $R_k$, we define 
$$N_k = \dsum_{j=0}^{\left\lfloor\frac{k}{2}\right\rfloor}\binom{k}{2j}N^{k-2j}M^{2j}+\frac{1}{n}\left(\dsum_{\lambda\vdash k,|\lambda|>1}\left(\dprod_{i>2}^{|\lambda|}\frac{c_{\lambda_i}}{n}\right)N^{\lambda_1-1}SN^{\lambda_2-1}\right)$$
and get
$$N_n = \dsum_{k=0}^{n-2}d_{n-k}N_k$$
As noted earlier, the difference of the $L_n$ and $R_n$ relation is $2M$ times the $L^kR^l$ relation, so here only the sum of the $L_n$ and $R_n$ relations is given.\\
While this set of generators is in a sense more natural than the one given above, with two of the generators being harmonic, it is also much more unwieldy algebraically as $N$ and $M$ do not have the symmetry that $L$ and $R$ possessed. A set of generators with all three generators being harmonic is even more complicated algebraically.\\
${ }$\\
Example: for $n = 4$, the $SN^kS$ relations look exactly like the $SL^kS$ relations, since $N = \frac{1}{2}(L+R)$ and $S$ doesn't distinguish between $L$ and $R$.
$$4N^3+4NM^2 = \frac{c_2}{2}N+\frac{c_3}{3}$$
$$N_1 = N, \text{ } N_2 = N^2+M^2 +\frac{1}{4}S, \text{ }  N_3 = N^3+3NM^2+\frac{1}{4}(NS+SN)$$
$$N_4 = N^4+6N^2M^2+M^4+\frac{1}{4}(N^2S+NSN+SN^2)+\frac{1}{16}c_2S$$
$$N^4+6N^2M^2+M^4+\frac{1}{4}(N^2S+NSN+SN^2)+\frac{1}{16}c_2S = \frac{c_2}{2}(N^2+M^2 +\frac{1}{4}S)+\frac{c_3}{3}N+\frac{c_4}{4}-\frac{c_2^2}{8}$$
The linearly independent pieces now look like
$$1, M, N, M^2, MN, N^2, S, M^3, M^2N, MN^2, NS, SN$$ 
$$M^4, M^3N, N^2S, NSN, SN^2, M^5, N^2SN, NSN^2, N^2SN^2$$
\section*{6: Generalized Exponents}
For each subrepresentation of $\mg\otimes \mg^*$ with highest weight $\lambda$, there is a projection operator $$P_\lambda (x_\alpha\otimes x^\beta) = Pr(\lambda)_{\alpha\delta}^{\beta\gamma}x_\gamma\otimes x^\delta \in V_\lambda \subset \mg\otimes \mg^*$$ 
calculated in [1]. The generalized exponents of $\lambda$ are then the degrees of the nonvanishing elements of the form 
$$P_\lambda(x_\alpha\otimes x^\beta)P_\beta^\alpha$$ 
that are linearly independent over $I(\mg)$.\\
Using the Killing form, we identify $\mg^*$ with $\mg$ and consider $\mg\otimes \mg$. As a $\mg$-module this decomposes into $\wedge^2\mg$ and $S^2\mg$, the alternating and symmetric tensor square respectively, which then further decompose into irreducible representations.\\
For $n = 2$, $\wedge^2\mg$ is isomorphic to $\mg$ itself, and hence is the adjoint representation, with generalized exponent $1$. $S^2\mg$ decomposes into a trivial representation and a $5$-dimensional representation, with generalized exponents $0$ and $2$ respectively.\\
For $n = 3$, $\wedge^2\mg$ decomposes into a copy $\mg$, with generalized exponents $1$ and $2$, and two dual $10$-dimensional representations with weights $3\omega_1$ and $3\omega_2$ respectively and each with generalized exponent $3$. $S^2\mg$ decomposes into the trivial representation with generalized exponent $0$, another copy of $\mg$, again with generalized exponents $1$ and $2$, and a $27$-dimensional representation with generalized exponents $2, 3$ and $4$. See [9] for details. Note that Rozhkovskaya uses a different basis, generated by harmonic elements. Her $M_1$ is proportional to $M$, her $N_1$ is proportional to $N$, and her $N_2$ is proportional to $3N^2+3M^2+S+c_2$.\\
For $n \geq 4$, the decomposition of $\mg\otimes\mg$ is fixed. $\wedge^2\mg$ decomposes into a copy of $\mg$ and two dual representations with highest weights $2\omega_1+\omega_{n-2}$ and $\omega_2+2\omega_{n-1}$ respectively, while $S^2\mg$ decomposes into the trivial representation, another copy of $\mg$, and two representations with highest weights $\omega_2+\omega_{n-2}$ and $2\omega_1+2\omega_{n-1}$ respectively. In an orthonormal basis for $\mg$, the corresponding elements of the family algebra are actually symmetric or antisymmetric as matrices.\\
The $L^kR^l$ reduction relation gives us a relation $\sim$ on elements in the $\omega_2+\omega_{n-2}$ representation; applying the differential operator $D = (\ptl^\alpha c_3)K_{\alpha\beta}\ptl^\beta$ gives a relation equivalent to the multiples of the $L^kR^l$ relation times $L+R$, modulo the $L_n$ and $R_n$ relations. Since $D$ transforms as the trivial representation, $D$ applied to both sides of $\sim$ again gives a relation between elements of the $\omega_2+\omega_{n-2}$ representation; hence we get that the generalized exponents of $\omega_2+\omega_{n-2}$ plus a copy of $\{n-1,\ldots,2n-2\}$ gives the generalized exponents of $2\omega_1+2\omega_{n-1}$.\\ 
Along with the fact that $\omega_1+\omega_{n-1}$ has generalized exponents $1,\ldots,n-1$ gives us enough information to get the full set of generalized exponents for the representations in question, given in table 1. Note that the two copies of $\omega_1+\omega_{n-1}$ each give an independent set of harmonic basis elements, one symmetric, one antisymmetric. Written as $q$-multiplicities, we get that the generalized exponents of representations of $sl(n)$ are equivalent to the Kostka polynomials, which are computable from Young Tableaux [2]. Hence we can easily check the results given.
\begin{table*}[h]
\caption{Generalized Exponents in $C_{\omega_1+\omega_{n-1}}(sl(n))$}
	\centering
		\begin{tabular}{|c|c|}
		\hline
		Weight & q-multiplicity\\ \hline\hline
		$0$& $1$\\ \hline
$\omega_1+\omega_{n-1}$& $q[n-1]_q$\\ \hline
$2\omega_1+\omega_{n-2}$& $q^3\binom{n-1}{2}_q$\\ \hline
$\omega_2+2\omega_{n-1}$& $q^3\binom{n-1}{2}_q$\\ \hline
$\omega_2+\omega_{n-2}$& $q^2\binom{n-1}{2}_q-q^{n-1}$\\ \hline
$2\omega_1+2\omega_{n-1}$& $q^2\binom{n}{2}_q$\\ \hline
		\end{tabular}
\end{table*}
${ }$\\
Example: for $n = 4$, we have the following generalized exponents:
\begin{table*}[h]
	\centering
		\begin{tabular}{|c|c|}
		\hline
		Weight & q-multiplicity\\ \hline\hline
		$0$& $1$\\ \hline
$\omega_1+\omega_3$& $q+q^2+q^3$\\ \hline
$2\omega_1+\omega_2$& $q^3+q^4+q^5$\\ \hline
$\omega_2+2\omega_3$& $q^3+q^4+q^5$\\ \hline
$\omega_2+\omega_2 = 2\omega_2$& $q^2+q^4$\\ \hline
$2\omega_1+2\omega_3$& $q^2+q^3+2q^4+q^5+q^6$\\ \hline
		\end{tabular}
\end{table*}
\end{document}